\documentclass[12pt]{article}

\usepackage[margin=1in]{geometry}
\usepackage{amsmath,amsthm,amssymb,amsfonts}
\usepackage{mathtools}
\usepackage{hyperref}
\usepackage{url}
\usepackage[numbers,sort&compress]{natbib}
\usepackage{enumitem}
\usepackage{silence}
\usepackage[expansion=false]{microtype}
\usepackage{graphicx}
\usepackage{tikz}
\usetikzlibrary{arrows.meta,positioning}

\theoremstyle{plain}
\newtheorem{theorem}{Theorem}[section]
\newtheorem{proposition}[theorem]{Proposition}
\newtheorem{lemma}[theorem]{Lemma}
\newtheorem{corollary}[theorem]{Corollary}

\theoremstyle{definition}

\theoremstyle{remark}
\newtheorem{remark}[theorem]{Remark}

\newcommand{\R}{\mathbb{R}}
\newcommand{\C}{\mathbb{C}}
\newcommand{\tr}{\operatorname{tr}}

\newcommand{\doi}[1]{\textsc{doi:}~\texttt{#1}}
\DeclareMathOperator{\dt}{det}

\DeclareMathOperator{\ord}{ord}
\DeclareMathOperator{\Newt}{Newt}
\DeclareMathOperator{\Covol}{Covol}
\DeclareMathOperator{\covol}{covol}
\DeclareMathOperator{\MV}{MV}
\newcommand{\ip}[2]{\langle #1,\, #2\rangle}
\newcommand{\Gp}{\Gamma_{\!+}}

\title{Newton geometry of Bogdanov--Takens degeneracies\\
	in integrable dilatonic models:\\
	invariant divisors and boundary multiplicity}

\author{E.\ Chan-L\'opez\textsuperscript{1}\thanks{Corresponding author:
		\texttt{eduardo.clopez13@gmail.com}.\;\;\;
		\textsc{orcid}:~0009-0003-7712-7907.}
	\and A.\ Mart\'in-Ruiz\textsuperscript{2}
	\and J. M.\ Paulin Fuentes\textsuperscript{1}}
\date{%
	\textsuperscript{1}Divisi\'on Acad\'emica de Ciencias B\'asicas,
	Universidad Ju\'arez Aut\'onoma de Tabasco, 86690 Cunduac\'an,
	Tabasco, Mexico\\
	\textsuperscript{2}Instituto de Ciencias Nucleares, Universidad Nacional
	Aut\'onoma de M\'exico, 04510 Ciudad de M\'exico, Mexico}

\begin{document}
	\maketitle
	
	\begin{abstract}
		The Kantowski--Sachs interior of Grumiller's two-dimensional dilaton
		gravity model, in which $H$ is the expansion rate of the orbit
		two-spheres, $u$ their inverse areal radius and $Q$ minus twice the
		Rindler acceleration, is the planar family
		$\dot H=\tfrac12(\Lambda-3H^{2}-u^{2})+Qu$, $\dot u=-Hu$, which
		carries, for $Q\neq0$, two distinct rank-one nilpotent equilibria with
		\emph{complementary}
		Bogdanov--Takens (BT) degeneracies: one with $a=0$, $b\neq0$ and local
		multiplicity $\mu=3$, one with $a\neq0$, $b=0$ and $\mu=2$.  We prove
		that neither degeneracy is accidental, and that the two are complementary manifestations of a single
		divisor-organized structure.  The family is Darboux integrable:
		$X=\tfrac12u^{4}X_{I}$, where $I$ is the mass function and $u^{4}$ is
		an inverse integrating factor whose zero divisor is exactly the
		invariant axis $\{u=0\}$.  Off that divisor, we show that
		\emph{every} nilpotent equilibrium of a vector field of the form
		$R\,X_{I}$ with $R(p)\neq0$ has $b=0$, by a two-line argument using
		the intrinsic identity $b=D_{q_{0}}\tr DX$.  On that divisor, we show
		that \emph{every} nilpotent equilibrium of a field with an invariant
		coordinate axis has $a=0$, because $\dt DX$ restricted to the axis
		factors as $\Phi_{x}\,G$ and both factors vanish at the corner.
		Consequently every nilpotent equilibrium of the class satisfies
		$ab=0$, so the two-parameter Bogdanov--Takens unfolding is never versal
		within it. The Bernstein--Kushnirenko bound does not apply at either point; we compute both multiplicities exactly instead, from the
		elementary local identity $\mu_{0}=m_{1}\ell_{y}+m_{2}\ell_{x}$, valid
		whenever one component is a monomial times a unit, requiring no
		nondegeneracy hypothesis and returning $\mu=3$ and $\mu=2$.  Both
		numbers admit a mixed-covolume reading, for which we give a closed
		planar formula $\Covol=\sum_{k,i}\min(p_{k}q_{i}',p_{i}'q_{k})$ on
		convenient diagrams; whether a general Bernstein--Kushnirenko
		multiplicity theorem holds for the \emph{non-convenient} diagrams that
		actually occur here is left open.  For the class
		$\dot H=\tfrac{1-m}{2}H^{2}+\psi(u)$, $\dot u=-Hu$ with $\psi(0)=0$,
		$\psi'(0)\neq0$ and $m\neq1$, the vacuum is always a nilpotent point
		with $\mu=3$, $a=0$, $b=-m$, and the discriminant that classifies its
		nilpotent type, in the form of Dumortier--Llibre--Art\'es, is the
		perfect square $b_{1}^{2}+8a_{3}=(m-2)^{2}$: the
		vacuum always has one hyperbolic and one elliptic sector, for every
		$m>1$ and independently of the potential $\psi$.
	\end{abstract}
	
	
	
	\section{Introduction}\label{sec:intro}
	
	The Bogdanov--Takens bifurcation \cite{Bogdanov1975,Takens1974} is the
	generic codimension-two bifurcation of a planar equilibrium with
	nilpotent linearization.  On the centre manifold the germ reduces to
	\begin{equation}\label{eq:BT}
		\dot\xi_{1}=\xi_{2},\qquad
		\dot\xi_{2}=a\,\xi_{1}^{2}+b\,\xi_{1}\xi_{2}+O(\|\xi\|^{3}),
	\end{equation}
	and the bifurcation is versal precisely when $ab\neq0$
	\cite{Kuznetsov2004}.  In a companion paper \cite{Companion} the two
	coefficients were characterized intrinsically: if $X$ is planar,
	$X(p)=0$, $J=DX(p)$ nilpotent of rank one and $q_{0}$ spans $\ker J$,
	then
	\begin{equation}\label{eq:intrinsic}
		a=-\tfrac12\,D_{q_{0}}\!\left(\dt DX\right)(p),
		\qquad
		b=D_{q_{0}}\!\left(\tr DX\right)(p),
	\end{equation}
	These identities are not confined to the plane.  In
	\cite{CompanionRn} they are generalized to $\R^{n}$: if $J=DX(p)$ has
	rank $n-1$ with $0$ an eigenvalue of algebraic multiplicity two, and
	$e_{k}(A)$ denotes the sum of the principal $k\times k$ minors of $A$,
	then the quadratic coefficients of the Bogdanov--Takens normal form on
	the centre manifold are
	\begin{equation}\label{eq:intrinsicRn}
		a=-\frac12\,\frac{D_{q_{0}}e_{n}}{e_{n-2}},
		\qquad
		b=\frac{D_{q_{0}}e_{n-1}}{e_{n-2}}
		-\frac{e_{n-3}\,D_{q_{0}}e_{n}}{e_{n-2}^{2}},
	\end{equation}
	the invariants being evaluated at $p$; for $n=2$ one has
	$e_{2}=\dt DX$, $e_{1}=\tr DX$, $e_{0}=1$ and $e_{-1}=0$, and
	\eqref{eq:intrinsicRn} collapses to \eqref{eq:intrinsic}.  Everything
	in the present paper takes place in the plane, so \eqref{eq:intrinsic}
	is what we use; we record \eqref{eq:intrinsicRn} because it fixes the
	sense in which $a$ and $b$ are \emph{directional derivatives of the
		characteristic invariants} rather than artefacts of a two-dimensional
	accident, and because the mechanisms isolated in
	Theorems~\ref{thm:B} and~\ref{thm:A} (a vanishing derivative of
	$\tr DX$ along $\ker J$, a factorization of $\dt DX$ along an invariant
	divisor) are statements about $e_{n-1}$ and $e_{n}$ that are
	meaningful in any dimension.
	
	Returning to the plane, $a\neq0$ is transversality of the kernel line $\ell$ to the fold
	set $\Sigma=\{\dt DX=0\}$, and $b\neq0$ is transversality of $\ell$ to
	the neutrality set $\mathcal H=\{\tr DX=0\}$.  The same work established
	a combinatorial obstruction: for Kolmogorov systems
	$X=(xA,yB)$ with $\MV(\Newt A,\Newt B)\le2$, every isolated
	\emph{interior} nilpotent equilibrium has local multiplicity $\mu=2$ and
	hence $a\neq0$; the nilpotent singularities of saddle, focus and
	elliptic type, which all require $\mu\ge3$, are unreachable.
	
	The present paper began as an application of \eqref{eq:intrinsic} to a
	gravitational family and turned into something else.  Consider the Kantowski--Sachs \cite{KantowskiSachs1966} interior of
	Grumiller's two-dimensional dilaton gravity model \cite{Grumiller2010},
	which in the variables of Section~\ref{sec:provenance} is the planar
	family
	\begin{equation}\label{eq:model}
		\dot H=\tfrac12\!\left(\Lambda-3H^{2}-u^{2}\right)+Qu,
		\qquad
		\dot u=-Hu,
		\qquad
		\lambda=(\Lambda,Q)\in\R^{2}.
	\end{equation}
	For $Q\neq0$ it carries two distinct rank-one nilpotent equilibria, and
	their BT data are \emph{complementary}:
	\begin{equation}\label{eq:table}
		\begin{array}{c|c|c|c|c}
			\text{point} & \text{parameters} & a & b & \mu\\\hline
			p_{0}=(0,0) & \Lambda=0 & 0 & -4 & 3\\
			p_{Q}=(0,Q) & \Lambda=-Q^{2} & Q/2 & 0 & 2
		\end{array}
	\end{equation}
	One point realizes the degeneracy $a=0$ forbidden by the mixed-volume
	obstruction of \cite{Companion}; the other realizes $b=0$.  This is the
	phenomenon the paper explains.
	
	Three structures do the explaining, and none of them is visible from the
	normal form alone.
	
	\smallskip
	\noindent\textbf{(i) The family is integrable.}  In
	Section~\ref{sec:integrable} we exhibit a Darboux first integral and an
	inverse integrating factor,
	\begin{equation}\label{eq:IintroA}
		I=\frac{H^{2}}{u^{3}}-\frac{\Lambda}{3u^{3}}-\frac{Q}{u^{2}}+\frac1u,
		\qquad
		X=\tfrac12 u^{4}\,X_{I},
		\qquad
		X_{I}=(\partial_{u}I,\,-\partial_{H}I),
	\end{equation}
	so that $X$ is a reparametrized Hamiltonian field away from $\{u=0\}$.
	Physically $I$ is the mass function, the conserved quantity every
	two-dimensional dilaton gravity possesses \cite{GKV2002}.
	
	\smallskip
	\noindent\textbf{(ii) Integrability forces $b=0$.}
	Theorem~\ref{thm:B} shows that if $X=R\,X_{I}$ and $p$ is a nilpotent
	equilibrium with $R(p)\neq0$, then $b=D_{q_{0}}\tr DX(p)=0$
	\emph{always}.  The proof is two lines from \eqref{eq:intrinsic}.  This
	accounts for $p_{Q}$, which lies off the divisor $\{u=0\}$.
	
	\smallskip
	\noindent\textbf{(iii) The invariant divisor forces $a=0$.}
	Theorem~\ref{thm:A} shows that if $\{y=0\}$ is invariant, i.e.
	$\dot y=y\,G$, then at a nilpotent equilibrium on that axis one has
	$\dt DX|_{y=0}=\Phi_{x}(x,0)\,G(x,0)$, a product of two factors that
	both vanish at the corner; hence $a=0$ and $\mu\ge3$ \emph{always}.
	This accounts for $p_{0}$.
	
	Together these give the paper's main structural statement,
	Theorem~\ref{thm:C}: every nilpotent equilibrium of the class satisfies
	$ab=0$, so the two-parameter unfolding \eqref{eq:BT} is never versal
	within the class.  We deliberately avoid the word ``codimension'' for
	this conclusion: what is proved is the vanishing of a product, not a
	transversality count, and the two points are in fact degenerate to very
	different degrees (Remark~\ref{rem:pQdegenerate}).  Which of the two
	degeneracies occurs is decided by a single divisor: the zero locus of
	the inverse integrating factor, which here coincides with the invariant
	axis.  The ``boundary'' referred to throughout is that divisor,
	$\{R=0\}$, and not the toric boundary of the ambient chart; the two
	should not be confused (Remark~\ref{rem:notinterior}).
	
	Section~\ref{sec:newton} turns to the combinatorial question, and here we
	are deliberately more cautious than the phenomenon invites.  The
	Bernstein--Kushnirenko bound \cite{Bernstein1975,Kouchnirenko1976} counts
	solutions in $(\C^{*})^{2}$ and is available at neither of our two
	points.  What we \emph{prove} is an elementary local identity
	(Proposition~\ref{prop:monomial}): if one component is a monomial times
	a unit, then $\mu_{0}=m_{1}\ell_{y}+m_{2}\ell_{x}$, where $\ell_{x}$ and
	$\ell_{y}$ are the intercepts of the Newton diagram of the other.  The
	identity is exact, needs no nondegeneracy hypothesis, and gives $\mu=3$
	and $\mu=2$ immediately.  What we then \emph{observe}
	(Section~\ref{sec:covolinterp}) is that both numbers coincide with a
	mixed covolume of Newton diagrams, the coconvex mixed volume of
	Khovanski\u\i--Timorin \cite{KhovanskiiTimorin2014}, for which we record
	a closed planar formula (Proposition~\ref{prop:minformula}).  This is an
	interpretation, not a theorem of this paper: the germs occurring at
	$p_{0}$ are \emph{not convenient}, so the local Bernstein--Kushnirenko
	statement we quote does not literally apply to them, and the coconvex
	theory of \cite{KhovanskiiTimorin2014} is developed for bounded coconvex
	bodies.  Establishing a general Bernstein--Kushnirenko multiplicity
	theorem for non-convenient Newton diagrams is left open.  With that
	caveat, the chain
	\[
	\text{Newton polytopes}\to\MV\to\mu\to\ord f\to a
	\]
	of \cite{Companion} has a plausible boundary counterpart with $\MV$
	replaced by the mixed covolume and the torus by the local ring.
	
	Section~\ref{sec:type} determines the topological type.  For the general
	class
	\begin{equation}\label{eq:generalclass}
		\dot H=\tfrac{1-m}{2}H^{2}+\psi(u),
		\qquad
		\dot u=-Hu,
		\qquad
		\psi(0)=0,\ \psi'(0)\neq0,\ m\neq1,
	\end{equation}
	which is integrable with inverse integrating factor $u^{m}$, the vacuum
	is automatically nilpotent with
	\[
	\mu=3,\qquad a=0,\qquad b=-m,
	\]
	and its quasi-homogeneous $3$-jet is
	$\dot x=y$, $\dot y=\tfrac{1-m}{2}x^{3}-m\,xy$.  The discriminant that
	decides its nilpotent type in the classification of
	Dumortier--Llibre--Art\'es \cite[Thm.~3.5]{DLA2006} is then the perfect
	square
	\begin{equation}\label{eq:square}
		b_{1}^{2}+8a_{3}=m^{2}+4(1-m)=(m-2)^{2}.
	\end{equation}
	Since a square is never negative, and since that classification covers
	the case of equality, for every $m>1$ the vacuum has one hyperbolic and
	one elliptic sector, independently of the potential $\psi$; the focus
	and centre cases are unreachable in the class.  The finite-codimension
	unfolding theory of \cite{DRS1991} is a separate matter, invoked only
	in Remark~\ref{rem:elliptic-stratum}: it assumes the strict inequality,
	so it identifies the codimension-three elliptic stratum for $m\neq2$
	alone.  For
	the physical value $m=4$ the discriminant equals $4$, and the elliptic
	sector is a one-parameter family of orbits leaving the vacuum and
	returning to it.  This is precisely one of the
	three types that the mixed-volume obstruction of \cite{Companion}
	excludes in the interior; it reappears, canonically, on the boundary.
	
	The question behind all of this is whether the nilpotent degeneracies of
	the dilatonic Kantowski--Sachs dynamics are accidents of particular
	parameter values or consequences of the structure of the theory.  They
	are structural.  The model carries a conserved mass function and an
	inverse integrating factor whose zero divisor is the invariant axis
	$\{u=0\}$, and these two structures constrain the intrinsic
	Bogdanov--Takens coefficients in complementary regions of the phase
	plane: off the divisor, integrability forces $b=0$; on it, the
	factorization of $\dt DX$ forces $a=0$.  Hence $ab=0$ throughout the
	class, and the generic two-parameter Bogdanov--Takens unfolding is
	unavailable within it.  At the vacuum the multiplicity $\mu=3$ and the
	perfect square \eqref{eq:square} turn that restriction into a definite
	phase portrait: an elliptic nilpotent singularity carrying a
	one-parameter family of trajectories that leave the static vacuum and
	return to it.  That homoclinic sector is the direct dynamical
	consequence of the integrable structure of the dilaton-gravity model.
	
	The remainder of the paper is organized as follows.
	Section~\ref{sec:model} derives the planar system from the underlying
	dilaton-gravity model and identifies its characteristic invariants and
	nilpotent equilibria.  Section~\ref{sec:integrable} establishes the
	Darboux-integrable structure through the conserved mass function and the
	inverse integrating factor.  Sections~\ref{sec:boffdivisor},
	\ref{sec:aoncorner} and~\ref{sec:dichotomy} isolate the two mechanisms
	responsible for the complementary degeneracies and combine them.
	Section~\ref{sec:newton} determines the local multiplicities from Newton
	diagrams and discusses their mixed-covolume interpretation.
	Section~\ref{sec:type} determines the local phase portrait of the vacuum
	and establishes its elliptic nilpotent character.
	Section~\ref{sec:algo} gives the symbolic-verification pipeline, and
	Section~\ref{sec:conclusions} summarizes the structural and physical
	consequences.
	
	\section{The model, its invariants, and two corrections}\label{sec:model}
	
	\subsection{Where the system comes from}\label{sec:provenance}
	
	The family \eqref{eq:model} is the Kantowski--Sachs reduction of
	Grumiller's two-dimensional dilaton gravity model \cite{Grumiller2010},
	and we record the derivation because the coefficient $Q$ is exactly the
	new coupling of that model.
	
	Grumiller's action, obtained by writing the most general
	spherically symmetric four-dimensional metric as
	$ds^{2}=g_{\alpha\beta}dx^{\alpha}dx^{\beta}
	+\Phi^{2}d\Omega^{2}$ and imposing power-counting renormalizability
	and analyticity on the resulting two-dimensional theory, is
	\begin{equation}\label{eq:grumaction}
		S=-\int d^{2}x\sqrt{-g}\,
		\bigl[\Phi^{2}R+2(\partial\Phi)^{2}
		-6\Lambda_{\mathrm G}\Phi^{2}+8a\Phi+2\bigr],
	\end{equation}
	depending on two constants: a cosmological constant
	$\Lambda_{\mathrm G}$ and a \emph{Rindler acceleration} $a$.  The term
	$8a\Phi$, linear in the dilaton, is the novel one.  A Birkhoff-like
	theorem holds, and in Schwarzschild gauge $\Phi=r$ the general solution
	is \cite{Grumiller2010}
	\begin{equation}\label{eq:grummetric}
		K^{2}(r)=1-\frac{2M}{r}-\Lambda_{\mathrm G}r^{2}+2ar .
	\end{equation}
	
	Inside a horizon, where $K^{2}<0$, the Killing field is spacelike and
	the metric takes the homogeneous Kantowski--Sachs form
	$ds^{2}=-dt^{2}+\alpha(t)^{2}dx^{2}+b(t)^{2}d\Omega^{2}$
	\cite{KantowskiSachs1966}, with $b$ the areal radius.  Put
	\begin{equation}\label{eq:physvars}
		H:=\frac{\dot b}{b},
		\qquad
		u:=\frac1b,
		\qquad
		\Lambda:=3\Lambda_{\mathrm G},
		\qquad
		Q:=-2a .
	\end{equation}
	Then $u$ is the square root of the Gaussian curvature of the orbit
	two-spheres, and $\dot u=-\dot b/b^{2}=-Hu$ identically: the second
	equation of \eqref{eq:model} is a kinematic identity, not a field
	equation.  The field equation in the $x$-direction reads
	\begin{equation}\label{eq:Gxx}
		2\frac{\ddot b}{b}+\Bigl(\frac{\dot b}{b}\Bigr)^{2}+\frac{1}{b^{2}}
		=\Lambda+\frac{2Q}{b},
	\end{equation}
	whose right-hand side is the only place the model departs from
	Einstein--$\Lambda$; substituting $\ddot b/b=\dot H+H^{2}$ and
	\eqref{eq:physvars} gives
	\[
	2\dot H+3H^{2}+u^{2}=\Lambda+2Qu,
	\]
	which is the first equation of \eqref{eq:model}.  Conversely, the first
	integral \eqref{eq:Iexplicit} of Section~\ref{sec:integrable} evaluates
	to $\dot b^{2}=2M/b-1+\Lambda b^{2}/3+Qb=-K^{2}(b)$ with $M$ the
	integration constant, which is the standard relation between the
	homogeneous interior and the static exterior and reproduces
	\eqref{eq:grummetric} under \eqref{eq:physvars}.  Thus $I=2M$, twice
	the conserved mass parameter of the static solution
	\eqref{eq:grummetric}; we do not identify $M$ with the
	four-dimensional Misner--Sharp mass, an identification that would not
	survive the presence of $\Lambda$ and $Q$.  The chain
	\[
	\text{action \eqref{eq:grumaction}}
	\to\text{Birkhoff solution \eqref{eq:grummetric}}
	\to\text{KS interior}
	\to\eqref{eq:model}
	\]
	closes.
	
	Two caveats, stated once and not repeated.  First, the identification
	$Q=-2a$ is with the Rindler acceleration of \cite{Grumiller2010}, a
	coupling of an effective long-distance model, not with an electric or
	dilatonic charge; a Maxwell field would contribute a term in $u^{4}$,
	not in $u$, and would leave the class \eqref{eq:class}.  Second, the
	mathematical results below (Theorems~\ref{thm:B}, \ref{thm:A},
	\ref{thm:C} and \ref{thm:elliptic}) use only the hypotheses stated in
	them and are independent of this interpretation; the reader who prefers
	to regard $Q$ as a free deformation parameter loses nothing.
	
	One point of domain, made once.  The physical Kantowski--Sachs
	interpretation has $b>0$, hence $u>0$; the mathematical analysis below
	is carried out on the full analytic $(H,u)$-plane, of which the sector
	$u>0$ is one component of the complement of the invariant axis.  The
	invariance of $\{u=0\}$ is a consequence of the second equation of
	\eqref{eq:model}, $\dot u=-Hu$, and of nothing physical.
	
	\subsection{Invariants}
	
	Write $X=(F_{1},F_{2})$ with
	\begin{equation}
		F_{1}(H,u)=\tfrac12\!\left(\Lambda-3H^{2}-u^{2}\right)+Qu,
		\qquad
		F_{2}(H,u)=-Hu.
	\end{equation}
	
	\begin{lemma}\label{lem:invariants}
		The Jacobian, trace and determinant of \eqref{eq:model} are
		\begin{equation}\label{eq:JTD}
			DX=\begin{pmatrix}-3H & Q-u\\[2pt] -u & -H\end{pmatrix},
			\qquad
			\tr DX=-4H,
			\qquad
			\dt DX=3H^{2}+Qu-u^{2},
		\end{equation}
		and therefore
		$\nabla\tr DX=(-4,0)$, $\nabla\dt DX=(6H,\;Q-2u)$.
	\end{lemma}
	
	\begin{proof}
		Direct differentiation;
		$\dt DX=(-3H)(-H)-(Q-u)(-u)=3H^{2}+u(Q-u)$.
	\end{proof}
	
	\begin{remark}[Two sign corrections]\label{rem:signs}
		It is worth recording two errata relative to the exploratory notes
		that motivated this work.  First, the determinant is
		$3H^{2}+Qu-u^{2}$ and \emph{not} $3H^{2}+u^{2}-Qu$; the two differ by
		the sign of the off-diagonal contribution.  Second, and as a
		consequence, the second coefficient at $p_{Q}$ is $a=+Q/2$, not
		$-Q/2$.  Neither correction affects the location of the nilpotent
		points, because on $\{H=0\}$ both expressions give
		$\dt DX(0,u)=u(Q-u)$; nor does it affect $a(p_{0})=0$, which is
		forced by $\partial_{H}\dt DX=6H$ vanishing at $H=0$.  All entries of
		\eqref{eq:table} other than the sign of $a(p_{Q})$ are unchanged.
	\end{remark}
	
	Note two facts that will organise everything below.  The neutrality set
	is the \emph{straight line}
	\begin{equation}
		\mathcal H=\{\tr DX=0\}=\{H=0\},
	\end{equation}
	and the axis $\{u=0\}$ is invariant, since $F_{2}=-Hu$ is divisible by
	$u$.  The axis $\{H=0\}$ is \emph{not} invariant, because $F_{1}$
	contains the term $Qu$.  The family therefore has exactly one invariant
	coordinate divisor: it is the minimal breaking of the Kolmogorov
	structure, and this turns out to be exactly what permits a nilpotent
	point at the corner (Remark~\ref{rem:kolmogorov}).
	
	\subsection{The nilpotent equilibria}\label{sec:equilibria}
	
	Equilibria satisfy $Hu=0$.  Throughout this subsection $Q\neq0$; for
	$Q=0$ the two points below coincide and the Jacobian at the origin is
	not the rank-one nilpotent matrix the Bogdanov--Takens theory requires.
	
	\smallskip
	\noindent\textbf{The vacuum.}  The origin is an equilibrium iff
	$\Lambda=0$.  For $\Lambda=0$, $Q\neq0$,
	\[
	J_{0}=\begin{pmatrix}0&Q\\0&0\end{pmatrix},
	\qquad
	\ker J_{0}=\R\,q_{0},\quad q_{0}=(1,0)^{T},
	\]
	which is nilpotent of rank one.  From Lemma~\ref{lem:invariants},
	$\nabla\dt DX(0,0)=(0,Q)$ and $\nabla\tr DX=(-4,0)$, so
	\eqref{eq:intrinsic} gives
	\begin{equation}
		a(p_{0})=-\tfrac12\ip{(0,Q)}{(1,0)}=0,
		\qquad
		b(p_{0})=\ip{(-4,0)}{(1,0)}=-4 .
	\end{equation}
	Observe that $q_{0}$ is the tangent direction of the \emph{invariant}
	axis $\{u=0\}$.
	
	\smallskip
	\noindent\textbf{The horizon branch.}  On $H=0$, $\dt DX=u(Q-u)$, so
	nilpotency off the vacuum forces $u=Q$, and the equilibrium condition
	$\Lambda-u^{2}+2Qu=0$ then forces $\Lambda=-Q^{2}$.  At
	$p_{Q}=(0,Q)$,
	\[
	J_{Q}=\begin{pmatrix}0&0\\-Q&0\end{pmatrix},
	\qquad
	\ker J_{Q}=\R\,q_{0},\quad q_{0}=(0,1)^{T},
	\]
	and $\nabla\dt DX(0,Q)=(0,-Q)$, whence
	\begin{equation}
		a(p_{Q})=-\tfrac12\ip{(0,-Q)}{(0,1)}=\frac Q2,
		\qquad
		b(p_{Q})=\ip{(-4,0)}{(0,1)}=0 .
	\end{equation}
	For $Q>0$ the point $p_{Q}$ lies in the physical sector $u>0$; for
	$Q<0$ it belongs to the analytic extension of the system.  Nothing
	below distinguishes the two cases.  Here $q_{0}$ is tangent to
	$\mathcal H=\{H=0\}$, which is why $b$ vanishes; this will be upgraded to a theorem in
	Section~\ref{sec:boffdivisor}.
	
	\begin{remark}[$p_{Q}$ is not an interior point either]\label{rem:notinterior}
		It is tempting to call $p_{Q}$ an ``interior'' equilibrium because
		$u=Q\neq0$.  This is only half true and the distinction matters.  In
		the original chart $p_{Q}$ lies on the coordinate line $\{H=0\}$, so
		it is not in the torus $(\C^{*})^{2}$; after the shift $v=u-Q$ it sits
		at the origin of the local chart, exactly like $p_{0}$.  What
		distinguishes the two points is not membership of the torus but
		membership of the \emph{invariant} divisor $\{u=0\}$, equivalently of
		the zero set of the inverse integrating factor.  Every statement below
		is organised by that divisor, not by the toric boundary of the
		ambient chart.
	\end{remark}
	
	\section{The family is Darboux integrable}\label{sec:integrable}
	
	\begin{proposition}\label{prop:integrable}
		Let $m\in\R$, let $\psi$ be analytic, and consider
		\begin{equation}\label{eq:class}
			X:\qquad
			\dot x=\Phi(x,y):=\tfrac{1-m}{2}x^{2}+\psi(y),
			\qquad
			\dot y=-xy .
		\end{equation}
		Then $V(x,y)=y^{m}$ is an inverse integrating factor for $X$, i.e.
		$\operatorname{div}(V^{-1}X)=0$, and
		\begin{equation}\label{eq:firstintegral}
			I(x,y)=\tfrac12 x^{2}y^{1-m}+\int^{y}\!\psi(s)\,s^{-m}\,ds
		\end{equation}
		is a first integral of $X$ on $\{y\neq0\}$.  Equivalently
		$X=V\,X_{I}$ with $X_{I}=(\partial_{y}I,-\partial_{x}I)$.
		Conversely, within the class $\dot y=-xy$, a monomial inverse
		integrating factor $y^{m}$ exists if and only if
		$\partial_{x}\Phi=(1-m)x$, i.e. if and only if $\Phi$ has the form
		\eqref{eq:class}.
	\end{proposition}
	
	\begin{proof}
		$\operatorname{div}X=\Phi_{x}+\partial_{y}(-xy)=(1-m)x-x=-mx$.  For
		$V=y^{m}$, $X\!\cdot\!\nabla V=(-xy)\,my^{m-1}=-mx\,V
		=V\operatorname{div}X$, which is exactly the inverse-integrating-factor
		equation.  For \eqref{eq:firstintegral}, $\partial_{x}I=xy^{1-m}$, so
		$V\,(-\partial_{x}I)=-xy=\dot y$; and
		$\partial_{y}I=\tfrac{1-m}{2}x^{2}y^{-m}+\psi(y)y^{-m}
		=\Phi\,y^{-m}$, so $V\,\partial_{y}I=\Phi=\dot x$.  The converse is
		the computation $X\!\cdot\!\nabla(y^{m})=V\operatorname{div}X
		\iff -mx=\Phi_{x}-x$.
	\end{proof}
	
	The model \eqref{eq:model} is the case $m=4$,
	$\psi(u)=\tfrac12(\Lambda-u^{2})+Qu$, and then
	\eqref{eq:firstintegral} becomes, after multiplication by $2$,
	\begin{equation}\label{eq:Iexplicit}
		I=\frac{H^{2}}{u^{3}}-\frac{\Lambda}{3u^{3}}-\frac{Q}{u^{2}}+\frac1u
		=\frac{3H^{2}-\Lambda-3Qu+3u^{2}}{3u^{3}},
		\qquad
		X=\tfrac12u^{4}\,X_{I}.
	\end{equation}
	By Section~\ref{sec:provenance}, $I=2M$ with $M$ the conserved mass
	parameter of the static solution \eqref{eq:grummetric}; that some such
	conserved quantity exists for \emph{every} two-dimensional dilaton
	gravity is standard \cite{GKV2002}, and
	Proposition~\ref{prop:integrable} identifies the exponent $m$ of the
	associated inverse integrating factor with the coefficient of $H^{2}$ in
	the Raychaudhuri equation.  The consistency check
	$\operatorname{div}X=-mH=-4H=\tr DX$ recovers Lemma~\ref{lem:invariants}.
	
	\section{Integrability forces $b=0$ off the divisor}\label{sec:boffdivisor}
	
	\begin{theorem}\label{thm:B}
		Let $U\subseteq\R^{2}$ be open, let $I,R$ be analytic on $U$
		(for the identity below $I,R\in C^{2}(U)$ suffices; analyticity is
		assumed so that $a$ and $b$ are the normal-form coefficients of
		\eqref{eq:BT}), and let
		\[
		X=R\,X_{I},
		\qquad
		X_{I}=\bigl(\partial_{y}I,\;-\partial_{x}I\bigr).
		\]
		Let $p\in U$ be an equilibrium of $X$ with $R(p)\neq0$, and suppose
		$J=DX(p)$ is nilpotent and nonzero, with $q_{0}$ spanning $\ker J$.
		Then
		\[
		b=D_{q_{0}}\bigl(\tr DX\bigr)(p)=0 .
		\]
		In particular no vector field admitting a nonvanishing inverse
		integrating factor near $p$ can have a Bogdanov--Takens point with
		$b\neq0$ at $p$.
	\end{theorem}
	
	\begin{proof}
		Since $R(p)\neq0$ and $X(p)=R(p)X_{I}(p)=0$, we get $X_{I}(p)=0$.
		Differentiating $X=R\,X_{I}$,
		\begin{equation}\label{eq:leibniz}
			DX=R\,DX_{I}+X_{I}\otimes\nabla R .
		\end{equation}
		Taking traces and using $\tr DX_{I}\equiv0$ (a Hamiltonian field is
		divergence free) gives the \emph{global} identity
		\begin{equation}\label{eq:trRXI}
			\tr DX=\ip{\nabla R}{X_{I}}\qquad\text{on }U .
		\end{equation}
		Differentiating \eqref{eq:trRXI} at $p$ and using $X_{I}(p)=0$,
		\[
		\nabla\bigl(\tr DX\bigr)(p)
		=\bigl(DX_{I}(p)\bigr)^{T}\nabla R(p).
		\]
		Hence, for any vector $q$,
		\[
		D_{q}\bigl(\tr DX\bigr)(p)
		=\ip{\bigl(DX_{I}(p)\bigr)^{T}\nabla R(p)}{q}
		=\ip{\nabla R(p)}{DX_{I}(p)\,q}.
		\]
		Finally, by \eqref{eq:leibniz} and $X_{I}(p)=0$ we have
		$J=DX(p)=R(p)\,DX_{I}(p)$ with $R(p)\neq0$, so
		$\ker DX_{I}(p)=\ker J$.  Choosing $q=q_{0}\in\ker J$ gives
		$DX_{I}(p)q_{0}=0$ and therefore $b=0$.
	\end{proof}
	
	\begin{corollary}\label{cor:pQ}
		Every nilpotent equilibrium of \eqref{eq:class} lying off the axis
		$\{y=0\}$ has $b=0$.  In particular $b(p_{Q})=0$ for the dilatonic
		model \eqref{eq:model}, for all $Q\neq0$.
	\end{corollary}
	
	\begin{proof}
		Apply Theorem~\ref{thm:B} with $R=V=y^{m}$, which is nonzero at any
		point with $y\neq0$, and $I$ from \eqref{eq:firstintegral}.
	\end{proof}
	
	\begin{remark}
		The mechanism is transparent in the language of \eqref{eq:intrinsic}.
		Identity \eqref{eq:trRXI} says that the neutrality set $\mathcal H$ of
		an integrable field is the zero set of $\ip{\nabla R}{X_{I}}$; at an
		equilibrium off $\{R=0\}$ the field $X_{I}$ vanishes, so $\mathcal H$
		is singular there in the direction of $\ker J$, and the kernel line
		$\ell$ is automatically tangent to it.  For the dilatonic model,
		$\nabla R=(0,2u^{3})$ and $X_{I}=(\partial_{u}I,-\partial_{H}I)$, so
		$\tr DX=-2u^{3}\partial_{H}I=-4H$, and $\mathcal H=\{H=0\}$ is exactly
		the locus where $I$ is critical in the $H$ direction.  The straightness
		of $\mathcal H$ noted in Section~\ref{sec:model} is a shadow of
		integrability, not a coincidence of the model.
	\end{remark}
	
	\subsection{Local consequence at $p_{Q}$}
	
	Set $x=v=u-Q$ so that $p_{Q}$ is the origin.  Then
	\begin{equation}\label{eq:localpQ}
		\dot H=-\tfrac32H^{2}-\tfrac12v^{2},
		\qquad
		\dot v=-H(Q+v).
	\end{equation}
	Since $\partial_{H}\dot v=-Q\neq0$, put $x=v$ and $y=\dot v=-H(Q+v)$,
	i.e. $H=-y/(Q+x)$.  A direct computation gives the Takens-type form
	\begin{equation}\label{eq:pQreduced}
		\dot x=y,
		\qquad
		\dot y=\underbrace{\tfrac{Q}{2}x^{2}+\tfrac12x^{3}}_{=:F(x)}
		+\frac{5}{2(Q+x)}\,y^{2}.
	\end{equation}
	Following the standard classification of nilpotent points
	\cite[Thm.~3.5]{DLA2006}, with $A\equiv0$, $f\equiv0$,
	$F(x)=\tfrac{Q}{2}x^{2}+\tfrac12x^{3}$ and
	$G(x)=\partial_{y}B(x,0)\equiv0$, the origin is a \emph{cusp} of
	multiplicity $\ord F=2$, and the identical vanishing of $G$ is the
	reason $b=0$.  Note that \eqref{eq:pQreduced} is of the form
	$\dot y=F(x)+h(x)y^{2}$, which linearises in $w=y^{2}$; this is the
	local avatar of the first integral \eqref{eq:Iexplicit}, analytic at
	$p_{Q}$ because $u=Q\neq0$ there.
	
	\begin{remark}[$p_{Q}$ is not the codimension-three cusp]
		\label{rem:pQdegenerate}
		The distinction between $G'(0)=0$ and $G\equiv0$ must not be blurred.
		The cusp of codimension three of \cite{DRS1987} is the germ with
		$\ord F=2$ and $G$ vanishing at the origin \emph{to order exactly
			two}; its generic three-parameter unfolding is the object of that
		paper.  At $p_{Q}$ the reduction \eqref{eq:pQreduced} has
		$G\equiv0$ identically, not merely $G'(0)=0$, so $p_{Q}$ is
		\emph{not} an instance of that classification.  The vanishing is not
		an artefact of the chart: by Proposition~\ref{prop:integrable} the
		germ at $p_{Q}$ admits an analytic first integral, and a germ with an
		analytic first integral imposes infinitely many independent conditions
		on its jets.  Topologically $p_{Q}$ is still a cusp; as a point of the
		space of germs it lies in the integrable stratum and is degenerate to
		infinite order in the direction that $b$ measures.  Accordingly we
		claim for it only what Theorem~\ref{thm:C} proves, namely $b=0$, and
		we attach no finite codimension to it.
	\end{remark}
	
	\section{The invariant divisor forces $a=0$ at the corner}\label{sec:aoncorner}
	
	\begin{theorem}\label{thm:A}
		Let $X=(\Phi,\Psi)$ be an analytic planar vector field near the origin
		for which the axis $\{y=0\}$ is invariant, i.e.
		$\Psi(x,y)=y\,G(x,y)$.  Assume $X(0)=0$, that $J=DX(0)$ is nilpotent
		and nonzero, and that the origin is an isolated equilibrium.  Then:
		\begin{enumerate}[label=\textup{(\roman*)},leftmargin=2.2em]
			\item $\Phi_{x}(0)=0$, $\Phi_{y}(0)\neq0$, $G(0)=0$, and
			$\ker J=\R\,(1,0)^{T}$; the kernel line is tangent to the
			invariant axis.
			\item Along the axis, $\dt DX(x,0)=\Phi_{x}(x,0)\,G(x,0)$, which
			vanishes at $x=0$ to order at least two.
			\item Consequently
			\[
			a=-\tfrac12\,D_{q_{0}}\bigl(\dt DX\bigr)(0)=0,
			\qquad
			b=\Phi_{xx}(0)+G_{x}(0).
			\]
			\item The local multiplicity satisfies
			\[
			\mu_{0}(X)=\ord_{x}\Phi(x,0)+I_{0}(\Phi,G)\;\ge\;3 ,
			\]
			where $I_{0}$ denotes the local intersection number at the
			origin.
		\end{enumerate}
	\end{theorem}
	
	\begin{proof}
		(i) $\Psi=yG$ has $D\Psi(0)=(0,G(0))$, so
		$J=\begin{psmallmatrix}\Phi_{x}(0)&\Phi_{y}(0)\\0&G(0)\end{psmallmatrix}$
		is upper triangular; nilpotency forces $\Phi_{x}(0)=G(0)=0$, and
		$J\neq0$ forces $\Phi_{y}(0)\neq0$.  Then $J$ has first row
		$(0,\Phi_{y}(0))$ and zero second row, so $\ker J=\{y=0\}$.
		
		(ii) $\dt DX=\Phi_{x}\,\partial_{y}(yG)-\Phi_{y}\,\partial_{x}(yG)
		=\Phi_{x}(G+yG_{y})-\Phi_{y}\,yG_{x}$.  Setting $y=0$ leaves
		$\Phi_{x}(x,0)G(x,0)$.  Both factors vanish at $x=0$ by (i), so the
		product vanishes to order $\ge2$.
		
		(iii) $q_{0}=(1,0)^{T}$ is tangent to $\{y=0\}$, so
		$D_{q_{0}}(\dt DX)(0)=\frac{d}{dx}\bigl[\dt DX(x,0)\bigr]_{x=0}=0$ by
		(ii); hence $a=0$.  For $b$: $\tr DX=\Phi_{x}+G+yG_{y}$, whose
		restriction to $y=0$ is $\Phi_{x}(x,0)+G(x,0)$, and differentiating at
		$x=0$ gives $\Phi_{xx}(0)+G_{x}(0)$.
		
		(iv) $\mu_{0}(X)=I_{0}(\Phi,\Psi)=I_{0}(\Phi,y\,G)
		=I_{0}(\Phi,y)+I_{0}(\Phi,G)$ by additivity of the intersection
		number in the second argument.  Now $I_{0}(\Phi,y)=\ord_{x}\Phi(x,0)$,
		which is $\ge2$ because $\Phi(0)=\Phi_{x}(0)=0$; and
		$I_{0}(\Phi,G)\ge1$ because both vanish at the origin.  Isolation of
		the equilibrium guarantees finiteness.
	\end{proof}
	
	\begin{corollary}\label{cor:p0}
		For the class \eqref{eq:class} with $\psi(0)=0$, $\psi'(0)\neq0$ and
		$m\neq1$, the origin is automatically an isolated nilpotent
		equilibrium of rank one, with
		\[
		a=0,\qquad b=-m,\qquad \mu_{0}=3 .
		\]
	\end{corollary}
	
	\begin{proof}
		Here $\Phi=\tfrac{1-m}{2}x^{2}+\psi(y)$ and $G(x,y)=-x$.  Then
		$\Phi_{x}(0)=0$ and $\Phi_{y}(0)=\psi'(0)\neq0$, so
		Theorem~\ref{thm:A} applies and $a=0$.  Also
		$\Phi_{xx}(0)=1-m$ and $G_{x}=-1$, so $b=(1-m)-1=-m$.  Finally
		$\ord_{x}\Phi(x,0)=\ord_{x}\bigl(\tfrac{1-m}{2}x^{2}\bigr)=2$ for
		$m\neq1$, and $I_{0}(\Phi,G)=I_{0}(\Phi,x)=\ord_{y}\psi(y)=1$ since
		$\psi'(0)\neq0$; hence $\mu_{0}=3$.
	\end{proof}
	
	\begin{remark}[Why $m\neq1$]\label{rem:mne1}
		The excluded value is not a Bogdanov--Takens resonance but a loss of
		isolation.  By Theorem~\ref{thm:A}(iii) the intercept jumps exactly
		when $\Phi_{xx}(0)=0$, i.e.\ when $b=G_{x}(0)$; within the class
		\eqref{eq:class} this forces $m=1$, hence $\Phi=\psi(y)$ and
		$\Phi(x,0)=\psi(0)=0$ identically.  Then $\ell_{x}=\infty$ and the
		whole axis $\{y=0\}$ consists of equilibria
		($\dot x=\psi(0)=0$, $\dot y=0$), so $\mu_{0}=\infty$ and the
		hypotheses of Theorem~\ref{thm:A} fail.  This is why $m\neq1$ is
		imposed here and $m>1$ in Section~\ref{sec:type}.
	\end{remark}
	
	For the dilatonic model, $m=4$ and $\psi(u)=\tfrac12(\Lambda-u^{2})+Qu$
	with $\psi(0)=\Lambda/2$; the vacuum condition $\Lambda=0$ is exactly
	$\psi(0)=0$, and $\psi'(0)=Q\neq0$.  We recover $a=0$, $b=-4$, $\mu=3$
	with no computation of the normal form.
	
	\begin{remark}[Why the Kolmogorov obstruction has no corner analogue]
		\label{rem:kolmogorov}
		For a genuine Kolmogorov system $\dot x=xP$, $\dot y=yQ$, the Jacobian
		at the origin is $\operatorname{diag}(P(0),Q(0))$, which is
		\emph{diagonal}; it can be nilpotent only if it is zero.  Hence a
		Kolmogorov system never has a rank-one nilpotent point at the corner,
		and the restriction to interior equilibria in \cite{Companion} is
		intrinsic rather than technical.  The dilatonic family breaks the
		Kolmogorov structure minimally (one invariant axis instead of two),
		and this is precisely enough to place a nilpotent point at the corner:
		the entry $\Phi_{y}(0)=\psi'(0)$ that makes $J$ nilpotent of rank one
		is the entry that a Kolmogorov system cannot have.
	\end{remark}
	
	\section{The dichotomy}\label{sec:dichotomy}
	
	\begin{theorem}\label{thm:C}
		Let $X$ be an analytic planar vector field on a neighbourhood of the
		origin satisfying
		\begin{enumerate}[label=\textup{(H\arabic*)},leftmargin=3em]
			\item the axis $\{y=0\}$ is invariant for $X$;
			\item $X=R\,X_{I}$ for some first integral $I$ and some
			$R$ with $\{R=0\}\subseteq\{y=0\}$.
		\end{enumerate}
		Let $p$ be an isolated equilibrium with $DX(p)$ nilpotent and nonzero.
		Then:
		\begin{enumerate}[label=\textup{(\roman*)},leftmargin=2.2em]
			\item if $p\notin\{y=0\}$, then $b=0$;
			\item if $p\in\{y=0\}$, then $a=0$ and $\mu_{p}(X)\ge3$.
		\end{enumerate}
		In either case $ab=0$; consequently the two-parameter unfolding
		\eqref{eq:BT} is not versal at $p$ within the class.
	\end{theorem}
	
	\begin{proof}
		(i) is Theorem~\ref{thm:B}, applicable because $R(p)\neq0$ by (H2).
		(ii) is Theorem~\ref{thm:A}, applicable by (H1) after translating $p$
		to the origin, which preserves $\{y=0\}$ since $p$ lies on it.
	\end{proof}
	
	\begin{corollary}\label{cor:nogenericBT}
		Every nilpotent equilibrium of a vector field satisfying (H1) and (H2),
		in particular of the dilatonic family \eqref{eq:model} for all
		$(\Lambda,Q)$, satisfies $ab=0$, and therefore is not a
		nondegenerate Bogdanov--Takens point.
	\end{corollary}
	
	\begin{remark}[What is and is not asserted]\label{rem:codim}
		Three cautions.  First, $ab=0$ is a statement about the vanishing of a
		product of two explicit derivatives; we do not translate it into a
		codimension, because the two points of \eqref{eq:model} are degenerate
		to different orders: $p_{0}$ satisfies the nondegeneracy conditions
		of the codimension-three elliptic stratum of \cite{DRS1991}
		(Theorem~\ref{thm:elliptic}), whereas $p_{Q}$ does not belong to any
		finite-codimension stratum of that classification
		(Remark~\ref{rem:pQdegenerate}).  Second, hypothesis (H2) is a genuine
		restriction and is \emph{not} implied by the mere existence of a first
		integral: at an equilibrium of $X=R\,X_{I}$ one has $R=0$ or
		$\nabla I=0$, so (H2) amounts to requiring that every equilibrium off
		$\{y=0\}$ be a critical point of $I$.  It must be checked model by
		model.  Third, the class \eqref{eq:class} is exactly the set of
		systems with $\dot y=-xy$ admitting a \emph{monomial} inverse
		integrating factor (Proposition~\ref{prop:integrable}); it is not all
		of two-dimensional dilaton gravity.
	\end{remark}
	
	\begin{remark}[Attribution]\label{rem:attribution}
		Neither Theorem~\ref{thm:B} nor Theorem~\ref{thm:A} is deep, and we do
		not claim either as a new result.  Theorem~\ref{thm:B} may be read as
		the conjunction of two elementary facts: a Hamiltonian planar field is
		divergence free, so $\tr DX_{I}\equiv0$ and $b$ vanishes trivially for
		it; and the vanishing of $b$ survives multiplication by a nonvanishing
		factor.  The proof given above is self-contained and does not appeal to
		orbital invariance, but the content is that observation.
		Theorem~\ref{thm:A} is the factorization $\dt DX|_{y=0}=\Phi_{x}G$
		together with the fact that both factors vanish at the corner.  The
		literature on inverse integrating factors and nilpotent singularities
		is extensive (see the survey \cite{GGGL2009} and, for the nilpotent
		case specifically, \cite{AlgabaGarciaGine2016}), and we have not
		found either statement recorded in the form used here; but both lie
		within easy reach of that literature, and a reader who locates a prior
		occurrence should regard the present contribution as the
		\emph{combination} of the two into the dichotomy of
		Theorem~\ref{thm:C}, and the consequences drawn from it in
		Sections~\ref{sec:newton} and~\ref{sec:type}.
	\end{remark}
	
	This is the structural statement the paper was looking for.  The two
	degeneracies in \eqref{eq:table} are not two accidents but the two
	strata of a single object: the divisor $\{u=0\}$, which is
	simultaneously the invariant axis of the flow and the zero locus of the
	inverse integrating factor.  Off it, integrability kills $b$; on it, the
	invariant axis kills $a$.  The ``boundary'' whose absence the toroidal
	argument was signalling is $\{R=0\}$.
	
	\section{Boundary multiplicity, and a Newton reading of it}\label{sec:newton}
	
	This section has two halves and the distinction between them is
	deliberate.  The first proves, by an elementary local identity, the exact
	multiplicities of the two nilpotent points; nothing there is conjectural
	and nothing there uses convexity.  The second observes that the two
	numbers coincide with mixed covolumes of Newton diagrams and explains why
	that is the natural combinatorial reading, while stopping short of a
	general theorem, for the reasons set out in
	Remark~\ref{rem:covolcaveat}.
	
	\subsection{An exact boundary multiplicity}\label{sec:exactmult}
	
	For $f\in\C\{x,y\}$ with $f(0)=0$ write
	\[
	\ell_{x}(f)=\ord_{x}f(x,0),
	\qquad
	\ell_{y}(f)=\ord_{y}f(0,y)
	\]
	for the intercepts of the Newton diagram of $f$ with the coordinate axes,
	with the convention $\ell=\infty$ when the restriction vanishes
	identically.
	
	\begin{proposition}\label{prop:monomial}
		Let $f\in\C\{x,y\}$ with $f(0)=0$ and let
		$g=x^{m_{1}}y^{m_{2}}\,U$ with $U(0)\neq0$.  Write
		$\ell_{x}=\ord_{x}f(x,0)$ and $\ell_{y}=\ord_{y}f(0,y)$ for the
		intercepts of the Newton diagram of $f$.  Then
		\begin{equation}\label{eq:monomialformula}
			I_{0}(f,g)=m_{1}\,\ell_{y}+m_{2}\,\ell_{x},
		\end{equation}
		No nondegeneracy hypothesis is required.
	\end{proposition}
	
	\begin{proof}
		Units do not change intersection numbers, and $I_{0}$ is additive in
		each argument, so
		$I_{0}(f,g)=m_{1}I_{0}(f,x)+m_{2}I_{0}(f,y)$.  Now $I_{0}(f,x)$ is the
		order of vanishing of $f(0,y)$, i.e. $\ell_{y}$, and symmetrically
		$I_{0}(f,y)=\ell_{x}$.  
	\end{proof}

	The hypothesis $U(0)\neq0$ is what makes this a boundary statement: it
	says that $g$ vanishes on the coordinate divisors and nowhere else near
	the origin, which is precisely the configuration a toroidal count cannot
	see.  The right-hand side is finite if and only if $f$ is divisible by
	neither $x$ nor $y$, that is, if and only if the origin is an isolated
	zero of the pair.
	
	\subsection{The two points}
	
	\noindent\textbf{The vacuum $p_{0}$, $\Lambda=0$.}  Here
	$f=F_{1}=-\tfrac32H^{2}-\tfrac12u^{2}+Qu$ and $g=F_{2}=-Hu$.  The
	support of $f$ is $\{(2,0),(0,2),(0,1)\}$, so its Newton diagram is the
	single segment from $(0,1)$ to $(2,0)$ and
	\[
	\ell_{x}=\ord_{H}F_{1}(H,0)=2,
	\qquad
	\ell_{y}=\ord_{u}F_{1}(0,u)=1 .
	\]
	Here $g$ is the monomial $H^{1}u^{1}$, so $m=(1,1)$ and
	Proposition~\ref{prop:monomial} gives
	\begin{equation}\label{eq:mu0}
		\mu_{p_{0}}=1\cdot\ell_{y}+1\cdot\ell_{x}=1+2=3 .
	\end{equation}
	Note where the two contributions come from: $\ell_{y}=1$ because $F_{1}$
	contains the \emph{linear} term $Qu$, which is the nonzero off-diagonal
	entry of the nilpotent Jacobian; $\ell_{x}=2$ because $F_{1}$ contains
	\emph{no} linear term in $H$, which is exactly the nilpotency condition
	$\tr DX(p_{0})=0$.  Thus the conjecture that ``absence of a linear term
	in $H$ implies $a=0$'' is now a theorem, in the following precise sense:
	$\ell_{x}\ge2$ is equivalent to $\Phi_{x}(0)=0$, and by
	Proposition~\ref{prop:monomial} it forces $\mu\ge3$, hence $a=0$ by the
	equivalence $a\neq0\iff\mu=2$ of \cite{Companion}.  This is a second
	route to Theorem~\ref{thm:A}(iii), independent of the factorization
	argument given there.
	
	\smallskip
	\noindent\textbf{The horizon point $p_{Q}$, $\Lambda=-Q^{2}$.}  In the
	local chart \eqref{eq:localpQ}, $f=-\tfrac32H^{2}-\tfrac12v^{2}$ and
	$g=-H(Q+v)$, so $g$ is the monomial $H^{1}v^{0}$ times a unit,
	$m=(1,0)$, and the intercepts of $f$ are $\ell_{x}=\ell_{y}=2$.  Hence
	\begin{equation}\label{eq:muQ}
		\mu_{p_{Q}}=1\cdot\ell_{y}+0\cdot\ell_{x}=2 .
	\end{equation}
	Both computations are two lines of order counting.  Neither used a
	Newton polytope, a mixed volume, or a nondegeneracy hypothesis, and the
	multiplicities in Table~\ref{tab:summary} should be regarded as
	established at this point, independently of everything in
	Section~\ref{sec:covolinterp}.
	
	\begin{table}[t]
		\centering
		\begin{tabular}{lccccc}
			\hline
			point & stratum & $a$ & $b$ & $\mu$ & controlling invariant\\
			\hline
			$p_{0}=(0,0)$, $\Lambda=0$ & $u=0$ (divisor) & $0$ & $-4$ & $3$
			& $\Covol=m_{1}\ell_{y}+m_{2}\ell_{x}=3$\\
			$p_{Q}=(0,Q)$, $\Lambda=-Q^{2}$ & $u\neq0$ & $Q/2$ & $0$ & $2$
			& $\Covol=2$\\
			\hline
		\end{tabular}
		\caption{The two rank-one nilpotent points of \eqref{eq:model}
			($Q\neq0$) and the local
			Newton invariant that computes their multiplicity.  The toroidal mixed
			volume of \cite{Bernstein1975} is unavailable at either point; the mixed
			covolume of Newton diagrams is exact at both.}
		\label{tab:summary}
	\end{table}

	\subsection{A mixed-covolume reading}\label{sec:covolinterp}
	
	For $f\in\C\{x,y\}$ with $f(0)=0$ let
	$\Gp(f)=\operatorname{conv}\bigl(\bigcup_{\alpha\in\operatorname{supp}f}
	(\alpha+\R^{2}_{\ge0})\bigr)$ be the Newton polyhedron, and let its
	\emph{Newton diagram} be the union of the compact faces of $\Gp(f)$.
	Call $f$ \emph{convenient} if the diagram meets both axes, and in that
	case set
	\begin{equation}
		\covol\Gp(f)=\operatorname{area}\bigl(\R^{2}_{\ge0}\setminus\Gp(f)\bigr)<\infty .
	\end{equation}
	The \emph{mixed covolume} of two convenient polyhedra is
	\begin{equation}\label{eq:covoldef}
		\Covol\bigl(\Gp(f),\Gp(g)\bigr)
		:=\covol\bigl(\Gp(f)+\Gp(g)\bigr)-\covol\Gp(f)-\covol\Gp(g),
	\end{equation}
	the Minkowski sum being taken on the $\Gp$'s.  Covolumes of scaled
	Minkowski sums are polynomial, so \eqref{eq:covoldef} is the
	polarization of $\covol$ in the sense of the theory of coconvex bodies
	\cite{KhovanskiiTimorin2014}; unlike ordinary mixed volumes, coconvex
	mixed volumes satisfy a \emph{reversed} Alexandrov--Fenchel inequality.
	This is the local counterpart of the mixed volume that appears in the
	Bernstein--Kushnirenko bound \cite{Bernstein1975,Kouchnirenko1976}, and
	it is the invariant that survives when the point of interest leaves the
	torus.
	
	In the plane the mixed covolume admits a closed formula, which we record
	because it makes the computations below mechanical.
	
	\begin{proposition}\label{prop:minformula}
		Let $f,g$ be convenient.  Orient the Newton diagram of $f$ from the
		$y$-axis to the $x$-axis and let its edge vectors, in that order, be
		$e_{k}=(p_{k},-q_{k})$ with $p_{k},q_{k}>0$, $k=1,\dots,r$; similarly
		$e_{i}'=(p_{i}',-q_{i}')$ for $g$, $i=1,\dots,s$.  Then
		\begin{equation}\label{eq:minformula}
			\Covol\bigl(\Gp(f),\Gp(g)\bigr)
			=\sum_{k=1}^{r}\sum_{i=1}^{s}\min\bigl(p_{k}q_{i}',\;p_{i}'q_{k}\bigr).
		\end{equation}
	\end{proposition}
	
	\begin{proof}
		Let $A=\sum_{k}q_{k}$ and $B=\sum_{k}p_{k}$ be the intercepts of the
		diagram of $f$.  Summing trapezoids under the diagram gives
		\[
		\covol\Gp(f)
		=\tfrac12\sum_{k}p_{k}q_{k}+\sum_{k<i}p_{k}q_{i},
		\]
		the indices being ordered by increasing slope of the edges.  The
		diagram of $\Gp(f)+\Gp(g)$ is obtained by concatenating all edges of
		both diagrams sorted by slope.  Applying the displayed formula to the
		merged list and subtracting the two individual covolumes cancels all
		pure $f$--$f$ and $g$--$g$ contributions (their relative order is
		preserved by the merge), leaving exactly the mixed pairs: the pair
		$(e_{k},e_{i}')$ contributes $p_{k}q_{i}'$ if $e_{k}$ is the steeper
		edge and $p_{i}'q_{k}$ otherwise.  Since $e_{k}$ steeper means
		$q_{k}/p_{k}>q_{i}'/p_{i}'$, i.e. $p_{i}'q_{k}>p_{k}q_{i}'$, both
		cases are the minimum of the two products.
	\end{proof}
	
	\begin{theorem}[local Bernstein--Kushnirenko; \cite{Kouchnirenko1976,KhovanskiiTimorin2014}]
		\label{thm:localBKK}
		For $f,g\in\C\{x,y\}$ convenient and without common factor through the
		origin,
		\[
		I_{0}(f,g)\;\ge\;\Covol\bigl(\Gp(f),\Gp(g)\bigr),
		\]
		with equality when the pair $(f,g)$ is nondegenerate with respect to
		the compact faces of $\Gp(f)$ and $\Gp(g)$.
	\end{theorem}

	When one germ is \emph{not} convenient (as happens whenever a
	component is divisible by a single monomial, which is the situation at
	both of our points), the two covolumes on the right of
	\eqref{eq:covoldef} are separately infinite while their difference need
	not be.  One may then define the mixed covolume by regularization,
	replacing $g$ by $g_{N}=g+x^{N}+y^{N}$.
	
	\begin{proposition}\label{prop:covolreg}
		Let $f$ have Newton diagram edge vectors $(p_{k},-q_{k})$ and let
		$g=x^{m_{1}}y^{m_{2}}U$ with $U(0)\neq0$.  Then for all sufficiently
		large $N$,
		\[
		\Covol\bigl(\Gp(f),\Gp(g_{N})\bigr)=m_{1}\ell_{y}+m_{2}\ell_{x},
		\]
		independently of $N$; that is, the regularized mixed covolume equals
		the exact multiplicity of Proposition~\ref{prop:monomial}.
	\end{proposition}
	
	\begin{proof}
		The diagram of $g_{N}$ consists of the two edges
		$(m_{1},-(N-m_{2}))$ and $(N-m_{1},-m_{2})$.  Formula
		\eqref{eq:minformula} gives, for $N$ large enough that the first edge
		is steeper and the second shallower than every $e_{k}$,
		\[
		\sum_{k}\Bigl[\min\bigl(p_{k}(N-m_{2}),\,m_{1}q_{k}\bigr)
		+\min\bigl(p_{k}m_{2},\,(N-m_{1})q_{k}\bigr)\Bigr]
		=\sum_{k}\bigl(m_{1}q_{k}+m_{2}p_{k}\bigr)
		=m_{1}\ell_{y}+m_{2}\ell_{x}. \qedhere
		\]
	\end{proof}
	
	Explicitly: at $p_{Q}$, regularizing $g$ to $H+v^{N}$ gives
	$\covol\Gp(f)=2$, $\covol\Gp(g_{N})=N/2$,
	$\covol(\Gp(f)+\Gp(g_{N}))=(N+8)/2$ and mixed covolume
	$(N+8)/2-2-N/2=2$; at $p_{0}$, regularizing $g=Hu$ to $Hu+H^{N}+u^{N}$
	gives $\covol\Gp(f)=1$, $\covol\Gp(g_{N})=N$,
	$\covol(\Gp(f)+\Gp(g_{N}))=N+4$ and mixed covolume $3$.
	
	\begin{remark}[What is claimed, and what is not]\label{rem:covolcaveat}
		Proposition~\ref{prop:covolreg} is a combinatorial identity; its
		agreement with Proposition~\ref{prop:monomial} is a fact relating
		those two statements, not a theorem about germs.  Three gaps separate
		it from a boundary Bernstein--Kushnirenko theorem, and we close none
		of them here.  (i)~Theorem~\ref{thm:localBKK}, which would supply the
		bridge, assumes both germs convenient; at $p_{0}$ the germ $g=Hu$ is
		not, so that theorem does not literally apply to the case of interest.
		(ii)~We have not shown $I_{0}(f,g_{N})=I_{0}(f,g)$: the regularization
		extends the combinatorics, it does not deform the germ.  (iii)~The
		coconvex mixed volumes of \cite{KhovanskiiTimorin2014} are constructed
		for \emph{bounded} coconvex bodies, whereas
		$\R^{2}_{\ge0}\setminus\Gp(Hu)$ is unbounded, so the identification
		with that theory is by analogy.  What can safely be said is that the
		exact boundary multiplicity admits a natural mixed-covolume reading;
		establishing a general Bernstein--Kushnirenko multiplicity theorem for
		non-convenient Newton diagrams is left open.
	\end{remark}
	
	\begin{remark}[What the torus loses]
		The Bernstein--Kushnirenko count is a statement about
		$(\C^{*})^{2}$; passing to a toric compactification distributes the
		total intersection among the boundary divisors, and an equilibrium
		sitting on such a divisor receives none of it.  What
		Proposition~\ref{prop:monomial} shows is that the lost information is
		entirely encoded in the \emph{intercepts} of the Newton diagram, i.e.
		in the position of the diagram relative to the coordinate axes, which
		is exactly the datum a mixed volume discards (mixed volumes are
		translation invariant; covolumes are not).  That is a reason to expect
		the mixed covolume, the translation-\emph{sensitive} polarization of
		the same construction, to be the right boundary object.  It is a
		reason, not yet a proof.
	\end{remark}
	
	\section{Topological type: the vacuum is elliptic}\label{sec:type}
	
	Corollary~\ref{cor:p0} gives $\mu=3$ at the vacuum, so the Takens
	reduction has $f(x)=c_{3}x^{3}+O(x^{4})$ with $c_{3}\neq0$ and the point
	is of saddle, focus or elliptic type \cite{DRS1991,DLA2006}.  We
	determine which.
	
	\begin{theorem}\label{thm:elliptic}
		Consider \eqref{eq:class} with $\psi(0)=0$, $\psi'(0)\neq0$ and
		$m>1$.  Then $(x,y)\mapsto(\xi,\eta):=\bigl(x,\Phi(x,y)\bigr)$ is a
		local diffeomorphism at the origin, its Jacobian determinant being
		$\Phi_{y}(0,0)=\psi'(0)\neq0$, and in these coordinates the germ
		takes the form
		\begin{equation}\label{eq:3jet}
			\dot\xi=\eta,
			\qquad
			\dot\eta=\underbrace{\tfrac{1-m}{2}}_{a_{3}}\xi^{3}
			\;\underbrace{-\,m}_{b_{1}}\,\xi\eta
			\;+\;\text{terms of quasi-homogeneous weight}\ \ge5,
		\end{equation}
		with weights $\operatorname{wt}(\xi)=1$,
		$\operatorname{wt}(\eta)=2$.  The
		sign of $b_{1}^{2}+8a_{3}$ is invariant under the scalings preserving
		\eqref{eq:3jet}, and
		\begin{equation}\label{eq:discriminant}
			b_{1}^{2}+8a_{3}=m^{2}+4(1-m)=(m-2)^{2}\;\ge\;0 .
		\end{equation}
		Consequently, since $a_{3}<0$ for $m>1$, the local phase portrait at
		the origin consists of one hyperbolic and one elliptic sector, for
		\emph{every} $m>1$ without exception.  The conclusion does not depend
		on $\psi$ beyond $\psi'(0)\neq0$.
	\end{theorem}
	
	\begin{proof}
		With $\xi=x$ and $\eta=\Phi(x,y)$ we have $\dot\xi=\eta$ and
		\[
		\dot\eta=\Phi_{x}\Phi+\Phi_{y}\dot y
		=(1-m)\xi\,\eta-\xi\,y\,\psi'(y).
		\]
		Since $\psi'(0)\neq0$, the equation
		$\eta=\tfrac{1-m}{2}\xi^{2}+\psi(y)$ determines $y=y(\xi,\eta)$ with
		$y=\psi'(0)^{-1}\bigl(\eta-\tfrac{1-m}{2}\xi^{2}\bigr)+O(2)$, whence
		$y\,\psi'(y)=\eta-\tfrac{1-m}{2}\xi^{2}+O(2)$ and
		\[
		\dot\eta=(1-m)\xi\eta
		-\xi\bigl(\eta-\tfrac{1-m}{2}\xi^{2}\bigr)+\cdots
		=\tfrac{1-m}{2}\xi^{3}-m\,\xi\eta+\cdots,
		\]
		the omitted terms having quasi-homogeneous weight $\ge5$.  That no
		term of weight $4$ occurs is not an accident of truncation: every term
		of $\dot\eta$ carries a factor $\xi$, and $y\,\psi'(y)$ is a power
		series in $y$ alone, while $y$ has weight~$2$ by the displayed
		inversion; hence $\dot\eta$ is $\xi$ times a series in a weight-two
		quantity and only the odd weights $3,5,7,\dots$ appear.  Under
		$\xi\mapsto\alpha\xi$, $\eta\mapsto\beta\eta$, $t\mapsto\gamma t$
		with $\beta\gamma=\alpha$ one has
		$a_{3}\mapsto\gamma^{2}\alpha^{2}a_{3}$ and
		$b_{1}\mapsto\gamma\alpha b_{1}$, so $b_{1}^{2}$ and $8a_{3}$ scale by
		the same positive factor and $\operatorname{sign}(b_{1}^{2}+8a_{3})$
		is invariant.  For the conclusion we apply the classification of
		nilpotent singular points in the form of \cite[Thm.~3.5]{DLA2006},
		which goes back to Andreev and to \cite{Dumortier1977}.  In its
		notation one writes the germ as $\dot\xi=\eta+A(\xi,\eta)$,
		$\dot\eta=B(\xi,\eta)$, lets $\eta=f(\xi)$ solve $\eta+A=0$ and sets
		$F(\xi)=B(\xi,f(\xi))=a\,\xi^{\mathfrak{m}}+o(\xi^{\mathfrak{m}})$,
		$G(\xi)=(\partial_{\xi}A+\partial_{\eta}B)(\xi,f(\xi))
		=b\,\xi^{\mathfrak{n}}+o(\xi^{\mathfrak{n}})$.  Our reduced germ
		\eqref{eq:3jet} has $A\equiv0$, hence $f\equiv0$, and
		\[
		F(\xi)=a_{3}\xi^{3}+\cdots,\qquad G(\xi)=b_{1}\xi+\cdots,
		\]
		so the dictionary is exact, with no rescaling and no sign change:
		$\mathfrak{m}=3$, $a=a_{3}$, $\mathfrak{n}=1$, $b=b_{1}$.  Theorem~3.5 of
		\cite{DLA2006}, case $\mathfrak{m}$ odd with $a<0$, then reads: if
		$\mathfrak{m}<2\mathfrak{n}+1$, or $\mathfrak{m}=2\mathfrak{n}+1$ and
		$b^{2}+4a(\mathfrak{n}+1)<0$, the origin is a centre or a focus; if
		$\mathfrak{n}$ is odd and either $\mathfrak{m}>2\mathfrak{n}+1$ or
		$\mathfrak{m}=2\mathfrak{n}+1$ and $b^{2}+4a(\mathfrak{n}+1)\ge0$, the local
		phase portrait consists of one hyperbolic and one elliptic sector.
		Here $\mathfrak{m}=3=2\mathfrak{n}+1$, $\mathfrak{n}=1$ is odd, and
		$b^{2}+4a(\mathfrak{n}+1)=b_{1}^{2}+8a_{3}$, which is
		\eqref{eq:discriminant}.  Since a square is never negative, the
		second alternative holds for every $m>1$.
	\end{proof}
	
	For the dilatonic model $m=4$: the vacuum has
	\[
	a_{3}=-\tfrac32,\qquad b_{1}=-4,\qquad b_{1}^{2}+8a_{3}=4>0,
	\]
	so it is an elliptic nilpotent point.  Explicitly, the reduction of
	\eqref{eq:model} at $\Lambda=0$ is
	\begin{equation}
		\dot x=y,
		\qquad
		\dot y=-\tfrac32x^{3}-4xy
		+\tfrac{1}{2Q^{2}}xy^{2}+\tfrac{3}{2Q^{2}}x^{3}y
		+\tfrac{9}{8Q^{2}}x^{5}+\cdots
	\end{equation}
	\begin{figure}[!t]
		\centering
		\includegraphics[width=\textwidth]{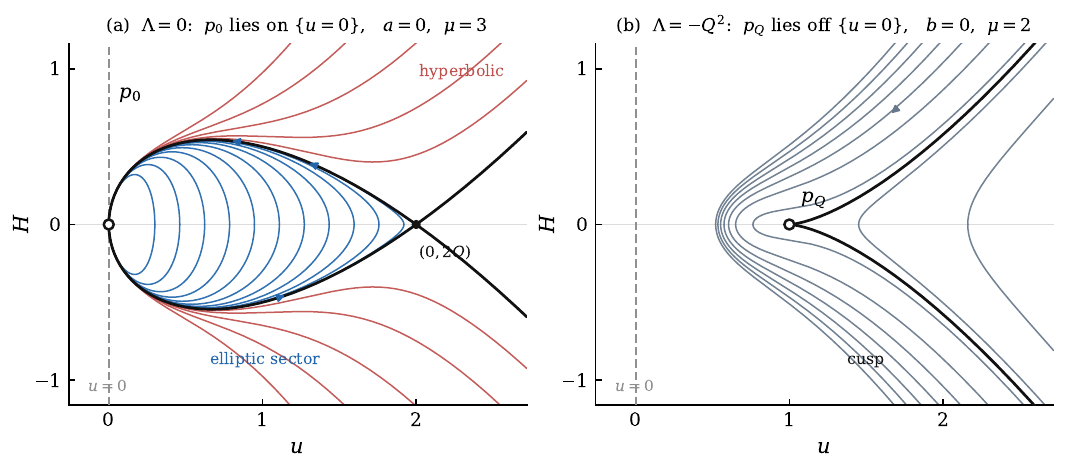}
		\caption{Orbits of \eqref{eq:model} for $Q=1$, drawn as exact level
			sets \eqref{eq:levelcubic} of the mass function
			\eqref{eq:Iexplicit}; nothing is integrated numerically.
			\textbf{(a)} $\Lambda=0$.  The level polynomial factors as
			$u\,(Cu^{2}-u+Q)$, so every orbit passes through the vacuum
			$p_{0}$: the blue curves are the one-parameter family of
			homoclinic trajectories forming the elliptic sector of
			Theorem~\ref{thm:elliptic}.  The separatrix $C=1/(4Q)$ (black)
			is the level on which the quadratic factor acquires a double
			root, at the hyperbolic saddle $(0,2Q)$; beyond it orbits escape
			(red), which is the hyperbolic sector.  \textbf{(b)}
			$\Lambda=-Q^{2}$.  The cusp at $p_{Q}$, off the divisor, sitting
			on the level $C=1/(3Q)$ whose double root is at $u=Q$; there
			$F_{1}(0,u)=-\tfrac12(u-Q)^{2}$, so the equilibrium is double,
			which is $\mu_{p_{Q}}=2$.  The dashed line is the invariant
			divisor $\{u=0\}$; only the physical sector $u>0$ is shown.}
		\label{fig:portraits}
	\end{figure}
	
	Geometrically, an elliptic point has one elliptic and one hyperbolic
	sector: there is a one-parameter family of orbits that leave the vacuum
	and return to it, alongside orbits that escape.  In the cosmological
	reading this is a family of solutions homoclinic to the static vacuum,
	the natural dynamical signature of the mass function
	\eqref{eq:Iexplicit} taking a degenerate critical value there.
	
	\subsection{The two portraits, drawn exactly}\label{sec:figure}
	
	That homoclinic family can be exhibited rather than described, and
	without integrating anything.  Because the model is integrable, its
	orbits are the level sets of the mass function, and fixing a level
	$I=C$ in \eqref{eq:Iexplicit} and solving for $H^{2}$ puts them in
	closed form:
	\begin{equation}\label{eq:levelcubic}
		H^{2}=C\,u^{3}+\frac{\Lambda}{3}+Qu-u^{2},
		\qquad C\in\R ,
	\end{equation}
	one cubic in $u$ for each orbit.  Figure~\ref{fig:portraits} is a plot
	of \eqref{eq:levelcubic}; every curve in it is an algebraic curve, and
	no trajectory has been integrated numerically.

	Three features of \eqref{eq:levelcubic} carry the content of the
	figure, and each is a one-line consequence of the closed form.
	
	At the vacuum, $\Lambda=0$ and the right-hand side factors as
	$u\,(Cu^{2}-u+Q)$, which vanishes at $u=0$ for \emph{every} $C$.  So
	every level curve passes through $p_{0}$: the homoclinic family of
	Theorem~\ref{thm:elliptic} is not something to be located numerically,
	it is the whole pencil of level sets, parametrized by $C$.  This is the
	elliptic sector, panel~(a).
	
	The family closes at $C=1/(4Q)$, where the discriminant of the
	quadratic factor vanishes and the double root sits at $u=2Q$.  That
	point is the second equilibrium of the vacuum branch: at $\Lambda=0$
	the equilibria are $(0,0)$ and $(0,2Q)$, and
	$\tr DX(0,2Q)=0$, $\dt DX(0,2Q)=-2Q^{2}<0$, so $(0,2Q)$ is a
	hyperbolic saddle.  Its stable and unstable separatrices are the
	crossing curves of panel~(a), and they separate the loops from the
	escaping orbits of the hyperbolic sector.
	
	At the horizon point, $\Lambda=-Q^{2}$ and the level through $p_{Q}$ is
	$C=1/(3Q)$, whose double root is at $u=Q$; this is the cusp,
	panel~(b).  Here the figure displays a multiplicity directly, since
	$F_{1}(0,u)=-\tfrac12(u-Q)^{2}$ has a double zero, so $p_{Q}$ is a
	double equilibrium and $\mu_{p_{Q}}=2$.  The number obtained from
	Newton intercepts in Section~\ref{sec:exactmult} can be read off the
	picture as the collision of two equilibria.
	
	Read together, the two panels are the dichotomy of
	Theorem~\ref{thm:C}: one family, two nilpotent points, two
	complementary degeneracies, separated by the divisor $\{u=0\}$ drawn
	dashed in both.
	
	\begin{remark}\label{rem:elliptic-stratum}
		The perfect square in \eqref{eq:discriminant} deserves emphasis.  The
		discriminant that separates the focus case from the
		hyperbolic-plus-elliptic case is generically an indefinite quadratic
		in the two normal-form coefficients; here it collapses to
		$(m-2)^{2}$, a function of the single number that measures the weight
		of the inverse integrating factor, and it is nonnegative for
		\emph{every} model in the class.  The focus and centre cases are
		therefore unreachable in this class: the vacuum always has an elliptic
		sector.
		
		The value $m=2$ is not an exception to that conclusion, but it is
		special in the desingularization.  The quasi-homogeneous blow-up
		$\xi=r$, $\eta=r^{2}z$ of \eqref{eq:3jet}, after division by $r$,
		gives
		$\dot r=rz$, $\dot z=a_{3}+b_{1}z-2z^{2}$, whose singular points on
		the exceptional divisor are the roots of $2z^{2}-b_{1}z-a_{3}=0$,
		with discriminant $b_{1}^{2}+8a_{3}=(m-2)^{2}$.  For $m\neq2$ these
		are two distinct hyperbolic points (a saddle and a node); at $m=2$
		they merge into a single saddle-node.  The topological type is
		unchanged (\cite[Thm.~3.5]{DLA2006} covers the equality case),
		but the finite-codimension unfolding theory of \cite{DRS1991}, which
		assumes the strict inequality, does not apply at $m=2$.  We therefore
		claim the topological statement for all $m>1$ and the identification
		with the codimension-three elliptic stratum only for $m\neq2$.
	\end{remark}
	
	\begin{remark}[The excluded type reappears at the boundary]
		Corollary~4.6 of \cite{Companion} states that for Kolmogorov systems
		with $\MV\le2$ the nilpotent singularities of saddle, focus and
		elliptic type cannot occur at isolated interior equilibria, since all
		of them require $\mu\ge3$; only cusps occur.  Theorem~\ref{thm:elliptic}
		shows that the elliptic type is exactly what one finds at the corner
		of a one-invariant-axis system.  The two results are therefore
		complementary rather than in conflict, and the boundary is where the
		forbidden types live.
	\end{remark}
	
	\section{Algorithm and symbolic verification}\label{sec:algo}
	
	Every statement above is checkable by a short symbolic computation, and
	we describe the pipeline in the order it should be run.
	
	\begin{enumerate}[leftmargin=2.2em]
		\item \textbf{Invariants.}  Compute $DX$, $\tau=\tr DX$,
		$\delta=\dt DX$ and their gradients.  This is $O(1)$ in the
		number of variables and requires no eigenvector.
		\item \textbf{Nilpotent locus.}  Solve
		$\{F_{1}=F_{2}=0,\ \tau=0,\ \delta=0\}$ together with
		$DX\neq0$.  For \eqref{eq:model} this is a $\operatorname{Groebner}$
		computation returning the two branches $\Lambda=0$ and
		$\Lambda=-Q^{2}$.
		\item \textbf{Coefficients.}  Apply \eqref{eq:intrinsic}: one null
		vector and two directional derivatives.
		\item \textbf{Integrating factor.}  Solve
		$X\!\cdot\!\nabla V=V\operatorname{div}X$ in the monomial ansatz
		$V=y^{m}$; this is a single linear condition on $m$.  Then
		integrate \eqref{eq:firstintegral}.
		\item \textbf{Multiplicity.}  Do \emph{not} compute
		$\dim_{\C}\C\{x,y\}/(F_{1},F_{2})$ by standard bases if the
		second component factors: use additivity,
		$I_{0}(F_{1},x^{m_{1}}y^{m_{2}}U)=m_{1}\ell_{y}+m_{2}\ell_{x}$,
		which costs two univariate order computations.  This is the
		practical payoff of Proposition~\ref{prop:monomial}: the local
		ring is never formed.
		\item \textbf{Type.}  Read $a_{3}$, $b_{1}$ off the quasi-homogeneous
		$3$-jet and evaluate $b_{1}^{2}+8a_{3}$.
	\end{enumerate}
	
	\noindent
	A \textsc{Wolfram Language} implementation is supplied as
	\texttt{BTDilatonNewton.wl}.  Its key design choices are worth naming.
	Steps~1--3 use \texttt{D} and \texttt{NullSpace} on exact symbolic
	matrices and never call \texttt{Eigenvectors}, which would introduce
	radicals and branch cuts for no benefit.  Step~2 uses
	\texttt{GroebnerBasis} with the saturation variable declared as an
	elimination variable, saturating by $Q-u\neq0$ via the Rabinowitsch
	trick, rather than \texttt{Solve}, so that the parameter stratification
	is returned as an ideal rather than as a list of special-cased branches.
	Step~5 does not perform a generic local-algebra computation.  Because the
	invariant divisor makes $F_{2}$ a monomial times a unit, the local
	intersection multiplicity is obtained from two univariate order
	computations, the coordinate-axis intercepts of
	Proposition~\ref{prop:monomial}, evaluated with
	\texttt{Exponent[\,\ldots,\ Min]}; this is not a programming
	optimization but a structural simplification dictated by the mathematics
	of the problem.  Step~6 uses \texttt{Series} with the
	quasi-homogeneous weighting imposed by hand, because
	\texttt{Normal[Series[...]]} with a single scaling parameter
	$\varepsilon$ under $x\to\varepsilon x$, $y\to\varepsilon^{2}y$ is the
	cheapest way to isolate the weight-$3$ part; \texttt{NormalForm}-style
	packages are unnecessary here and would obscure which coefficient is the
	invariant one.
	
	For comparison, the same symbolic pipeline can be assembled in other
	computer algebra systems, but the local-algebra route is more involved.
	The standard computation of an isolated multiplicity requires a
	local-algebra or equivalent zero-dimensional quotient calculation,
	whereas the present implementation replaces that step by two univariate
	order computations using the factorization of $F_{2}$.  The latter is not
	a software-specific shortcut: it is the structural calculation dictated
	by the divisor, and the same structure that makes the theorems hold is
	what makes the computation cheap.
	
	\section{Conclusions}\label{sec:conclusions}
	
	The dilatonic family \eqref{eq:model} is not merely a new place to apply
	the intrinsic identities \eqref{eq:intrinsic}.  It is a family in which
	the two ways of losing Bogdanov--Takens nondegeneracy are both forced,
	by two different structures that turn out to be two views of one
	divisor:
	\[
	\{u=0\}
	\;=\;
	\{\text{invariant axis of the flow}\}
	\;=\;
	\{\text{zeros of the inverse integrating factor }u^{4}\}.
	\]
	Off the divisor, integrability forces $b=0$ (Theorem~\ref{thm:B}); on
	it, the factorization of $\dt DX$ forces $a=0$ and $\mu\ge3$
	(Theorem~\ref{thm:A}).  Every nilpotent equilibrium of the class
	therefore satisfies $ab=0$, so the generic two-parameter
	Bogdanov--Takens unfolding is unavailable within it
	(Corollary~\ref{cor:nogenericBT}); and the geometric reason is the
	coincidence displayed above, that the zero divisor of the inverse
	integrating factor is simultaneously an invariant divisor of the flow.
	
	The two boundary multiplicities are $\mu_{p_{0}}=3$ and
	$\mu_{p_{Q}}=2$, computed exactly by \eqref{eq:monomialformula}, an
	elementary identity that needs no nondegeneracy hypothesis; that they
	also equal mixed covolumes of Newton diagrams, computable by
	\eqref{eq:minformula}, is an interpretation, and a general
	Bernstein--Kushnirenko multiplicity theorem for non-convenient Newton
	diagrams is not established here (Remark~\ref{rem:covolcaveat}).  At the
	vacuum, $\mu=3$ together with
	\[
	b_{1}^{2}+8a_{3}=(m-2)^{2}\;\ge\;0
	\]
	places the germ in the hyperbolic-plus-elliptic case of the
	classification for every $m>1$, with no exception
	(Theorem~\ref{thm:elliptic}).  For the physical value $m=4$ the
	discriminant equals $4$, and the elliptic sector consists of a
	one-parameter family of trajectories that leave the static vacuum and
	return to it, a family of solutions homoclinic to the vacuum of the
	dilaton-gravity model.
	
	Several extensions suggest themselves: broader classes of
	inverse-integrating-factor divisors, in which $\{R=0\}$ need not be a
	single coordinate axis and one expects the nilpotent locus to stratify
	by the order of vanishing of $R$; higher boundary multiplicities;
	versality within the restricted deformation space of integrable fields
	with a prescribed invariant divisor, rather than in the full space of
	vector fields; and a general Bernstein--Kushnirenko theory for
	non-convenient Newton diagrams, for which the reversed
	Alexandrov--Fenchel inequality of coconvex bodies
	\cite{KhovanskiiTimorin2014} suggests that the boundary statement should
	be a lower bound on multiplicity.  These lie beyond the scope of the
	present work.
	
	The scope of the two hypotheses should be stated with the same care.
	For the family \eqref{eq:model} both are verified explicitly.  (H1)
	holds because $\dot u=-Hu$, so the divisor $\{u=0\}$ is invariant; on
	the physical branch $b>0$ one has $u>0$, while the analysis is carried
	out on the full analytic $(H,u)$-plane.  (H2) holds because the inverse
	integrating factor computed in Section~\ref{sec:integrable} is
	$R=u^{4}/2$, whose zero set is that same divisor.  What does not
	transfer automatically to other models is (H2): the existence of a
	conserved mass function gives a first integral, but not the
	localization $\{R=0\}\subseteq\{y=0\}$, which must be checked case by
	case (Remark~\ref{rem:codim}).
	
	For the model at hand the chain is complete.  Grumiller's dilaton
	gravity fixes the Kantowski--Sachs system \eqref{eq:model}; that system
	carries a conserved mass function and an invariant divisor $\{u=0\}$;
	those two structures force $ab=0$; hence no versal two-parameter
	Bogdanov--Takens point exists in the class; and at the vacuum the
	resulting degeneracy is realized as an elliptic nilpotent singularity
	with a one-parameter family of homoclinic trajectories.  The physical
	consequence is therefore structural rather than accidental.  The absence
	of a versal Bogdanov--Takens point is not the result of tuning a
	parameter to a special value: it follows from the integrability and
	invariant-divisor structure that the Kantowski--Sachs dynamics inherits
	from the dilaton-gravity model.  The qualitative phase portrait of the
	vacuum is thus constrained by the same geometric structure that supplies
	the conserved mass.
	
	\section*{Acknowledgements}
	E. Chan-L\'opez acknowledges support from SECIHTI through the
	``Estancias Posdoctorales por M\'exico'' program (CVU 422090). A. Mart\'in-Ruiz acknowledges
	financial support from UNAM-PAPIIT (project IG100224), UNAM-PAPIME
	(project PE109226), SECIHTI (project CBF-2025-I-1862), and the Marcos
	Moshinsky Foundation.
	
	\section*{Ethics declarations}
	\subsection*{Conflict of interest}
	The author declares no conflict of interest.
	
	\subsection*{Data availability}
	No datasets were generated or analysed during this study.  The
	\textsc{Wolfram Language} verification script described in
	Section~\ref{sec:algo} is provided as supplementary material; it is
	self-contained and requires no further software.
	

\end{document}